\documentclass[12pt]{amsart}

\usepackage{amsmath,amssymb,amsbsy,amsfonts,amsthm,latexsym,
                        amsopn,amstext,amsxtra,euscript,amscd,mathrsfs,color,bm,cite}
                       
\usepackage{algorithm}
\usepackage{algorithmic}
\usepackage{float} 
\usepackage[english]{babel}
\usepackage{mathtools}
\usepackage{todonotes}
\usepackage[colorlinks,linkcolor=black,anchorcolor=black,citecolor=black,urlcolor=black]{hyperref}

\newtheorem{theorem}{Theorem}
\newtheorem{lemma}[theorem]{Lemma}

\newtheorem{corollary}[theorem]{Corollary}

\newtheorem{remark}[theorem]{Remark}

\numberwithin{equation}{section}
\numberwithin{theorem}{section}
\numberwithin{table}{section}
\numberwithin{figure}{section}

\def\cG{{\mathcal G}}
\def\cI{{\mathcal I}}
\def\cJ{{\mathcal J}}
\def\cP{{\mathcal P}}
\def\cQ{{\mathcal Q}}
\def\cS{{\mathcal S}}

\def\({\left(}
\def\){\right)}

\def\rf#1{\left\lceil#1\right\rceil}
\def\mand{\qquad \mbox{and} \qquad}

\def \F{{\mathbb F}}

\begin{document}

\title{Functional graphs of families of quadratic polynomials}

\author[B. Mans]{Bernard Mans}
\address{School of Computing, Macquarie University, Sydney, NSW 2109, Australia}
\email{bernard.mans@mq.edu.au}

\author[M. Sha]{Min Sha}
\address{School of Mathematical Sciences, South China Normal University, Guangzhou, 510631, China}
\email{min.sha@m.scnu.edu.cn}

\author[I. E. Shparlinski]{Igor E. Shparlinski}
\address{School of Mathematics and Statistics, University of New South Wales, Sydney, NSW 2052, Australia}
\email{igor.shparlinski@unsw.edu.au}

\author[D. Sutantyo]{Daniel Sutantyo}
\address{School of Computing, Macquarie University, Sydney, NSW 2109, Australia}
\email{daniel.sutantyo@gmail.com}

\begin{abstract} 
We study functional graphs generated by several 
quadratic polynomials, acting simultaneously on a finite field of odd characteristic.  
We obtain several results about the number of leaves in such graphs. In particular, in 
the case of graphs generated by three polynomials, we relate the distribution 
of leaves to the Sato-Tate distribution of Frobenius traces of 
elliptic curves. We also present extensive numerical results which we hope may 
shed some light on the distribution of leaves for larger families of polynomials. 
\end{abstract}



\maketitle

\section{Introduction}
\label{sec:intro}   

Let $\F_q$ be the finite field of $q$ elements, where $q$ is a power of some odd prime $p$. 
Recall that the functional graph of a polynomial $f \in \F_q[X]$ is the directed graph $\cG(f)$ with nodes labelled by 
the elements of $\F_q$ and with an edge from $u$ to $v$,   $u,v\in \F_q$ if and only if $f(u) = v$. 
In recent years, functional graphs over finite fields and their associated dynamical systems have been extensively studied, 
see~\cite{BaBr,  KLMMSS,MSSS1,MSSS2, SoKr, VaSha} and~\cite{BGHKST, BrGar1, BrGar2, BuSch, ChouShp, FlGar, Gart1, Gart2, H-B, Juul, JKMT, MaTh, OstSha, PomShp, Ren}
respectively.  

The case of quadratic polynomials 
has received a special attention, see~\cite{H-B, MSSS1, VaSha}. 
It is easy to see that 
the functional graph of a quadratic polynomial $aX^2 + bX + c$ is isomorphic to the functional graph of some polynomial $X^2 + d$ 
(see, for example, the proof of~\cite[Theorem~2.1]{KLMMSS}). 
Hence, up to isomorphism, it is enough to study functional graphs of polynomials of the form $f(X) = X^2 + a$. 
In particular, we see that  $(q-1)/2$ vertices $v \in \F_q$ in the graph $\cG(X^2 + a)$ for which $v-a \not \in \cQ_q$ are leaves (that is, of in-degree $0$). 
Here, $\cQ_q = \{u^2:~u\in \F_q\}$ is the set of $(q+1)/2$  squares in $\F_q$. 

We now generalise the construction of functional graphs to several polynomials. Namely, given $n$ polynomials 
$f_1, \ldots, f_n \in \F_q[X]$ we define the   directed graph $\cG(f_1, \ldots, f_n)$ with nodes labelled by 
elements of $\F_q$ and with an edge from $u$ to $v$, $u,v\in \F_q$, if and only if $ f_i(u) = v$ for some $i =1, \ldots, n$. 

In this paper, we continue studying the case of quadratic polynomials. 
For simplicity, we study the graphs $\cG(X^2 + a_1, \ldots,  X^2 + a_n)$ with pairwise distinct $a_1 , \ldots, a_n \in \F_q$. 
In particular, we are interested in the leaves of such graphs. 
Note that usually a directed graph having less leaves is more likely to have a large cycle. 

We remark that for a fixed $n$, the above graphs considered in this paper might not cover all the graphs generated by arbitrary $n$ 
pairwise distinct quadratic polynomials. 

Some of our motivations also come from considering covering $\F_q$ by shifted copies 
$\cQ_q+a = \{u+a:~u \in \cQ_q\}$ of the set of squares. 
Recall that the {\it Paley graph\/} of $\F_q$, denoted by $\cP_q$, is defined to be the graph with nodes labelled by   
the elements of $\F_q$ and with an edge from $u$ to $v$,   $u,v\in \F_q$ if and only if $v - u$ is a non-zero square. 
It is easy to see that the domination number of the Paley graph $\cP_q$ is exactly the minimal number of  
shifted copies $\cQ_q+a$ whose union is equal to $\F_q$. 
When $q=p$ is prime and $q \equiv 3 \pmod{4}$, this number is between 
$$
\Bigl(\frac{1}{2}-\varepsilon\Bigr)\frac{\log q}{\log 2} \mand  \Big\lfloor \frac{\log (q+1)}{\log 2} \Big\rfloor
$$ 
for any fixed $\varepsilon > 0$ and sufficiently large $q$; see~\cite[Theorems~1 and~3]{Lee}.   
Note that in~\cite{Lee} the definition of the Paley graph is different from the above, there it has an edge from $u$ to $v$ 
if and only if $u-v$ is a non-zero square, but this does not change the estimate we just mentioned.  We also mention the recent work~\cite{MSW} where a related but 
independent question has been considered. 

We remark that throughout the paper,
the implied constants in the `$O$' symbols are absolute. 

We state the main theoretical results in Section~\ref{sec:main} 
and prove them one by one in Section~\ref{sec:proofs} after collecting some preliminary results in Section~\ref{sec:preliminaries}. 
We present the numerical results in Section~\ref{sec:numerical}.

\section{Main results}   
\label{sec:main}

\subsection{Graphs with or without leaves}  

 First we show that if $n$ is not very large, then any graph $\cG(X^2 + a_1, \ldots, X^2 + a_n)$ 
 with pairwise distinct $a_1, \ldots, a_n \in \F_q$ has a leaf.

\begin{theorem}
\label{thm:Leaves}  
For any positive integer $n$ satisfying
$$
n^2 2^{4n-4} < q
$$ 
and arbitrary pairwise distinct elements $a_1, \ldots, a_n \in \F_q$, 
 the  directed graph $\cG(X^2 + a_1, \ldots, X^2 + a_n)$   has a  leaf. 
\end{theorem} 

When $n=2$ and $q \ge 64$, by Theorem~\ref{thm:Leaves}, for any $a_1 \ne a_2 \in \F_q$ the  directed graph $\cG(X^2 + a_1, X^2 + a_2)$ 
 has  a leaf. 

We remark that the condition for $n$ in Theorem~\ref{thm:Leaves} can be reformulated as 
$$
n \le \Bigl( \frac{1}{4} - \varepsilon\Bigr)\frac{\log q} {\log 2}
$$
for any fixed $\varepsilon > 0$ and sufficiently large $q$. 

When $q \equiv 3 \pmod{4}$, we obtain the following improvement upon Theorem~\ref{thm:Leaves}. 

\begin{theorem}
\label{thm:Leaves2}  
If $q \equiv 3 \pmod{4}$, then for any positive integer $n$ satisfying
$$
n^2 2^{2n-2} < q
$$ 
and arbitrary pairwise distinct $a_1, \ldots, a_n \in \F_q$, 
 the  directed graph $\cG(X^2 + a_1, \ldots, X^2 + a_n)$   has a  leaf. 
\end{theorem}

Next we show that there is a positive integer $n$ which is not very large and some 
 pairwise distinct $a_1, \ldots, a_n \in \F_q$  such that the graph $\cG(X^2 + a_1, \ldots, X^2 + a_n)$ 
has  no leaf.

\begin{theorem}
\label{thm:NoLeaves}   
There exist a positive integer $n$ satisfying
$$
n \le \left\lceil \frac{\log q}{\log 2} \right\rceil
$$ 
and some pairwise distinct $a_1, \ldots, a_n \in \F_q$  such that  
the   directed graph $\cG(X^2 + a_1, \ldots, X^2 + a_n)$   has no leaf. 
\end{theorem}

We remark that alternatively we can phrase  the conclusion of Theorem~\ref{thm:NoLeaves} as the existence
of $a_1, \ldots, a_n \in \F_q$ with 
$$
\bigcup_{i=1}^n  \(\cQ_q+a_i\)  =\F_q. 
$$

\subsection{Distribution of leaves in graphs} 

Given $n$ pairwise  distinct elements $a_1, \ldots, a_n \in  \F_q$ we  denote by 
$L_q(a_1, \ldots, a_n)$ the number of leaves in the graph $\cG(X^2 + a_1, \ldots, X^2 + a_n)$.  
Here we present some estimates about $L_q(a_1, \ldots, a_n)$. 

Clearly, $L_q(a) = (q-1)/2$ for any $a \in \F_q$.  
We also obtain an explicit formula for $L_q(a,b)$ for any $a,b \in \F_q$, see Theorem~\ref{thm:Count 2} below. 

Let $\chi$ be the quadratic character of $\F_q$, that is, 
\begin{equation}
\label{eq:Def eta}
\chi(u) = \begin{cases} 1, & \text{if}\ u \in \cQ_q^{*},\\
0, & \text{if}\ u=0, \\
-1, & \text{if}\ u \not \in \cQ_q, 
\end{cases} 
\end{equation}
where $\cQ_q^*= \cQ_q\setminus \{0\}$ is the set of non-zero squares
(that is, of quadratic residues).

\begin{theorem}
\label{thm:Count 2}  
Write  $q = 4k + 1$ or $4k + 3$ for some integer $k$. For any two distinct elements $a,b \in \F_q$, we have 
$$
L_q(a,b) =  k+  
\begin{cases}     
\(\chi(a-b) - 1 \)/2, &   \text{ if } q \equiv 1 \pmod 4,\\
0,  &    \text{ if } q \equiv 3 \pmod 4.
\end{cases} 
$$
 \end{theorem}

 For $n\ge 3$ we do not  have a simple and explicit formula, but we can obtain an 
 asymptotic formula for $L_q(a_1, \ldots, a_n)$. 
 
 \begin{theorem}
\label{thm:Count n}  
For any integer $n \ge 1$ and any $n$  pairwise distinct elements   $a_1, \ldots, a_n \in  \F_q$,  we have 
$$
\left|  L_q(a_1, \ldots ,a_n) - 2^{-n}q \right| <  \frac{2}{3} n q^{1/2}, 
$$
and for $n < 2q^{1/2}$, the above factor $2/3$ can be replaced by $1/2$. 
 \end{theorem}

 \subsection{Connection with elliptic curves}
 
For the case $n=3$, we indeed can give a formula for  $L_q(a, b, c)$ by relating it to the elliptic curve defined by $a, b, c$. 

For any pairwise distinct elements $a, b, c \in \F_q$, we define the elliptic curve over $\F_q$: 
\begin{equation}
\label{eq:E abc}
E_{a,b,c}: \quad Y^2 = (X-a)(X-b)(X-c), 
\end{equation}
 see~\cite{Silv} for a general background on elliptic curves.
Let $E_{a,b,c}(\F_q)$ be the set of $\F_q$-rational poins on $E_{a,b,c}$. 
Let 
\begin{equation}
\label{eq:trace}
t_{a,b,c} = q + 1 - \#E_{a,b,c}(\F_q) 
\end{equation}
be the \textit{Frobenius trace}  of $E_{a,b,c}$. 

Then, we obtain a formula for $L_q(a, b, c)$ in terms of $t_{a,b,c}$.  

 \begin{theorem}
\label{thm:Count 3}  
For any  pairwise distinct elements   $a,b,c \in  \F_q$,  we have 
\begin{align*}
&L_q(a,b,c)  -  \frac{q}{8}  \\
& \quad \ =  \frac{1}{8}\( t_{a,b,c} - 7 \)+ \frac{1}{8}\(1+\chi(-1)\) \(\chi(a-b) + \chi(a-c) + \chi(b-c) \) \\
& \qquad  \  \ -\frac{1}{8}\( \chi((a-b)(a-c)) + \chi((b-a)(b-c))  + \chi((c-a)(c-b))\). 
\end{align*}
 \end{theorem}

 Using Theorem~\ref{thm:Count 3} and the Hasse bound~\eqref{eq:Hasse}, we obtain the following improvement upon Theorem~\ref{thm:Count n} when $n=3$. 
 
 \begin{corollary}
\label{cor:Count 3}  
For any  pairwise distinct elements   $a,b,c \in  \F_q$,  we have 
$$
\left|  L_q(a, b, c) - \frac{q}{8}\right| \le  \frac{1}{4} q^{1/2} + 2. 
$$
 \end{corollary}

 Finally, for $n=3$ we have a result which in some sense interpolates between 
 the explicit formula of Theorem~\ref{thm:Count 2} and the upper bound of Theorem~\ref{thm:Count n}.  
To formulate the result, we now define the relative error
$$
\Delta_q(a,b,c) = \frac{L_q(a,b,c) -q/8}{q^{1/2}} .
$$
and then define the distribution function 
\begin{align*}
\rho_q(\sigma, \tau) &  = \frac{1}{q(q-1)(q-2)}
\# \Bigl\{(a,b,c) \in \F_q^3:\\
&\qquad \qquad \qquad  a\ne b, \  a\ne c, \ b \ne c, \ 
\Delta_q(a,b,c) \in [\sigma, \tau]\Bigr\}.
\end{align*}

We present our next result about $\rho_q(\sigma, \tau)$ essentially through the so-called Sato-Tate density (see~\eqref{eq:ST dens}), 
but only for primes $q=p>3$ since it is 
based on a work of Michel~\cite{Mich} which treats only this case.
However we believe that the result can be extended to arbitrary $q$ with $\gcd(q,6) = 1$.

 \begin{theorem}
\label{thm:Count 3+}  
Let $p>3$ be prime. Then,  uniformly over $[\sigma, \tau] \subseteq [-1/4, 1/4]$, we have 
$$
\rho_p(\sigma,\tau) = \frac{2}{\pi}\int_{4\sigma}^{4\tau}
(1-z^2)^{1/2}\, {\rm d} z+ O\bigl(p^{-1/4}\bigr). 
$$ 
\end{theorem}

\section{Preliminaries}     
\label{sec:preliminaries} 

\subsection{Intersections of shifted  squares}

The questions we consider in this paper are closely related to the intersections of values sets
$\{f(u):~u \in \F_q\}$ of polynomials  $f(X) = X^2 + a \in \F_q[X]$ (that is, intersections between the shifted sets $\cQ_q + a$ of squares). 
We have the following lemma on the cardinality of the intersection of shifted  squares, 
\begin{lemma}
\label{lem:ltersect Qa} 
For any pairwise distinct elements $a_1, \ldots, a_n \in \F_q$, we have 
$$
\left|\# \left( \bigcap_{i=1}^n (\cQ+a_i) \right) - \frac{q}{2^n} \right| < \frac{1}{2}n q^{1/2}. 
$$
\end{lemma}

\begin{proof}
The case $n=1$ is trivial. 
In the sequel, we assume $n \ge 2$. 

Recall that $\chi$ is the quadratic character of $\F_q$, see~\eqref{eq:Def eta}. 
Note that for any $u \in \bigcap_{i=1}^n (\cQ+a_i)$, we have that $u-a_i$ is a square in $\F_q$ for any $i=1, \ldots, n$. 
Then, by definition, we have 
\begin{align*}
& \# \Bigl( \bigcap_{i=1}^n (\cQ+a_i) \Bigr)  \\
& \qquad  \qquad= \frac{1}{2^n}\sum_{u \in \F_q \setminus\{a_1, \ldots, a_n\}}\bigl(\chi(u-a_1)+1\bigr) \cdots \bigl(\chi(u-a_n)+1\bigr)  \\ 
& \qquad\qquad  \qquad\quad+ \frac{1}{2^{n-1}}\sum_{i=1}^{n} \bigl(\chi(a_i - a_1)+1\bigr) \cdots \bigl(\chi(a_i - a_n)+1\bigr) \\
& \qquad  \qquad= \frac{1}{2^n}\sum_{u \in \F_q }\bigl(\chi(u-a_1)+1\bigr) \cdots \bigl(\chi(u-a_n)+1\bigr)  \\ 
&\qquad \qquad  \qquad \quad+ \frac{1}{2^{n}}\sum_{i=1}^{n} \bigl(\chi(a_i - a_1)+1\bigr) \cdots \bigl(\chi(a_i - a_n)+1\bigr). 
\end{align*}
By multiplying out the binomials in the first summation above and using the multiplicativity of $\chi$, we obtain 
\begin{align*}
 \# \Bigl( \bigcap_{i=1}^n (\cQ+a_i) \Bigr)  &= \frac{q}{2^n} +  \frac{1}{2^n}\biggl(\sum_{u \in \F_q }\chi(u-a_1) + \cdots  + \sum_{u \in \F_q }\chi(u-a_n) \biggr) \\ 
& \qquad   + \frac{1}{2^n}\sum_{1 \le i < j\le n}  \sum_{u \in \F_q }\chi\bigl((u-a_i)(u-a_j)\bigr) + \cdots  \\
& \qquad  + \frac{1}{2^n} \sum_{u \in \F_q }\chi\bigl((u-a_1)\cdots (u-a_n)\bigr) \\ 
& \qquad  + \frac{1}{2^{n}}\sum_{i=1}^{n} \bigl(\chi(a_i - a_1)+1\bigr) \cdots \bigl(\chi(a_i - a_n)+1\bigr). 
\end{align*}

Now, we estimate the last summation in the above. 
Recall that $q$ is an odd prime power.  
When $q \equiv 3 \pmod{4}$, we know that $-1$ is not a square in $\F_q$, 
and so for any $i \ne j$, if $\chi(a_j - a_i) = 1$ then $\chi(a_i - a_j) = -1$, and then we have 
\begin{equation} \label{eq:sum-chi1}
\frac{1}{2^{n}}\sum_{i=1}^{n} \bigl(\chi(a_i - a_1)+1\bigr) \cdots \bigl(\chi(a_i - a_n)+1\bigr) \le \frac{1}{2^{n}} \cdot 2^{n-1} = \frac{1}{2},
\end{equation}
since the sum on the left-hand side has at most one non-zero term. 
Indeed, if  this sum has one non-zero term, then for some $j \in \{1, \ldots, n\}$ we have
$$
\chi(a_j - a_1)= \cdots =\chi(a_j - a_n) = 1.
$$
Then, as we have noticed, for any $i\ne j$ we have $\chi(a_i - a_j) = -1$
and the corresponding term vanishes. 

When $q \equiv 1 \pmod{4}$, we know that $-1$ is a square in $\F_q$, and so 
$$
\frac{1}{2^{n}}\sum_{i=1}^{n} \bigl(\chi(a_i - a_1)+1\bigr) \cdots \bigl(\chi(a_i - a_n)+1\bigr) = \frac{1}{2^{n}}\cdot k \cdot 2^{n-1} = \frac{k}{2}, 
$$
where $k$ is the number of non-vanishing terms in the sum on the left-hand side. 
Without loss of generality, we assume that $k \ge 2$ and the non-vanishing terms correspond to $i=1,2, \ldots, k$. 
Then, $a_i - a_j$ is a square for any $1 \le i, j\le k$, which means that 
$a_1, \ldots, a_k$ form a clique in the Paley graph $\cP_q$. 
Using the trivial upper bound on the clique number of $\cP_q$ (see~\cite[Lemma~1.2]{Yip}), 
we have $k \le q^{1/2}$, which implies in this case 
\begin{equation} \label{eq:sum-chi2}
\frac{1}{2^{n}}\sum_{i=1}^{n} \bigl(\chi(a_i - a_1)+1\bigr) \cdots \bigl(\chi(a_i - a_n)+1\bigr) \le \frac{1}{2}q^{1/2}. 
\end{equation}

Hence, combining~\eqref{eq:sum-chi1} with~\eqref{eq:sum-chi2}, we always have 
$$
\frac{1}{2^{n}}\sum_{i=1}^{n} \bigl(\chi(a_i - a_1)+1\bigr) \cdots \bigl(\chi(a_i - a_n)+1\bigr) \le  \frac{1}{2}q^{1/2}.  
$$ 
In addition, note that 
$$\sum_{u\in \F_q} \chi(u) = 0,
$$ 
see~\cite[Equation~(5.37)]{LiNi}.  
Then, these, together with the Weil bound (see~\cite[Theorem~11.23]{IwKow} or~\cite[Theorem~5.41]{LiNi}), give 
\begin{align*}
& \left|\# \( \bigcap_{i=1}^n (\cQ+a_i) \) - \frac{q}{2^n} \right| \\
& \qquad  \qquad \le \frac{1}{2^n} \left(\binom{n}{2} + \binom{n}{3}2 + \cdots +  \binom{n}{n}(n-1)\right) q^{1/2} + \frac{1}{2}q^{1/2} \\ 
& \qquad  \qquad  = \frac{1}{2^n} \left(\binom{n}{1} + \binom{n}{2}2 + \cdots +  \binom{n}{n}n\right) q^{1/2}  \\ 
& \qquad \qquad  \qquad  - \frac{1}{2^n} \left(\binom{n}{1} + \binom{n}{2} + \cdots +  \binom{n}{n}\right) q^{1/2} + \frac{1}{2}q^{1/2} \\ 
& \qquad  \qquad  = \frac{n2^{n-1}}{2^n} q^{1/2} - \frac{2^n -1}{2^n}q^{1/2} + \frac{1}{2}q^{1/2}  \\
& \qquad  \qquad  = \(\frac{n}{2} + \frac{1}{2^n} - \frac{1}{2}  \) q^{1/2} < \frac{1}{2}n q^{1/2}, 
\end{align*}
where the last inequality follows from $n \ge 2$. 
This completes the proof. 
\end{proof}

We remark that the estimate in Lemma~\ref{lem:ltersect Qa} is somehow well-known; 
see, for example,~\cite[Corollary~5]{Per} for the case when $q$ is a prime 
(the proof there extends to any odd prime power $q$ without any changes). 
But the error term in~\cite[Corollary~5]{Per} is $n(3+q^{1/2})$, which is worse than the one in Lemma~\ref{lem:ltersect Qa}.

\subsection{Intersections of arbitrary sets and shifted  squares}

One of our main technical tools is the following result which may be of independent interest.

\begin{lemma}
\label{lem:Intesect} 
For any subset $\cS \subseteq \F_q$ of cardinality $\# \cS = s$, there exists some element $a \in \F_q$
such that 
$$
\#\(\cS \cap \(\cQ_q+a\) \) \ge   \frac{1}{2}s .
$$
\end{lemma}

\begin{proof} 
We recall that  $\chi$ denotes the quadratic character of $\F_q$, see~\eqref{eq:Def eta}.  Clearly 
\begin{align*}
\#\(\cS \cap (\cQ_q+a) \)  & \ge \#\(\cS \cap (\cQ_q^*+a) \) \\
&  =  \sum_{u \in \cS} \frac{1}{2}\(\chi(u-a) + 1\) = \frac{1}{2} s + \frac{1}{2} \sum_{u \in \cS} \chi(u-a).
\end{align*}
Therefore, changing the order of summation, we derive 
\begin{align*}
\sum_{a\in \F_q} \#\(\cS \cap (\cQ_q^*+a) \)  &=  \frac{1}{2}s q  + \frac{1}{2}\sum_{a\in \F_q}  \sum_{u \in \cS} \chi(u-a) \\
&= \frac{1}{2}s q + \frac{1}{2}\sum_{u \in \cS}   \sum_{a\in \F_q} \chi\( u-a \) =    \frac{1}{2}sq.
\end{align*}
The result now follows. 
\end{proof}

\subsection{Sato--Tate distributions in families of elliptic curves}

Let $p> 3$ be a prime number.

By a classical result of Birch~\cite{Birch} we know that the 
Frobenius traces of  elliptic curves in  the {\it Weierstrass form\/}
(defined similarly to~\eqref{eq:trace}), 
$$
Y^2 = X^3 +  aX + b
$$
is distributed in accordance with the Sato--Tate 
density when  $a$ and $b$  vary over all elements of $\F_p$ with $4a^2 +27 b^3 \ne 0$. 

However here we need a similar result for the family of curves 
$E_{a,b,c}$ given 
by~\eqref{eq:E abc}.  

First,  we recall that the famous~\textit{Hasse bound} (see~\cite[Chapter~V, Theorem~1.1]{Silv}) asserts that for any pairwise distinct $(a,b,c) \in \F_p^3$ we have
\begin{equation}
\label{eq:Hasse}
\left| t_{a,b,c} \right| \le 2q^{1/2}. 
\end{equation}
Hence, we can define the {\it Frobenius angle\/} $\psi_{a,b,c}\in [0, \pi]$ by the identity
\begin{equation}
\label{eq:ST angle}
\cos \psi_{a,b,c} = \frac{t_{a,b,c}}{2 p^{1/2}}.
\end{equation}

We now recall the definition 
of  the {\it Sato--Tate  density\/} which for $[\alpha,\beta] \subseteq [0,\pi]$
is given by 
\begin{equation}
\label{eq:ST dens}
\mu_{\tt ST}(\alpha,\beta) = \frac{2}{\pi}\int_\alpha^\beta
\sin^2\vartheta\, {\rm d} \vartheta = \frac{2}{\pi}\int_{\cos \beta}^{\cos \alpha}
(1-z^2)^{1/2}\, {\rm d} z,
\end{equation}
where $[\alpha,\beta] \subseteq [0,\pi]$.

\begin{lemma}
\label{lem:S-T abc}  Let $p>3$ be prime. Then,  uniformly over $[\alpha, \beta] \subseteq [0, \pi]$, we have
\begin{align*}
 \# \left\{(a,b,c) \in \F_p^3:~a\ne b, \  a\ne c, \ b \ne c, \ \psi_{a,b,c} \in [\alpha, \beta]\right\} &\\
= \mu_{\tt ST}(\alpha, \beta) p^3 &
+ O\(p^{11/4}\).
\end{align*}
\end{lemma}

\begin{proof}  
Making the change of variable $X\to X+c$, we see 
that it is enough to show that  
$$
\# \left\{(a,b) \in\bigl( \F_p^*\bigr)^2:~a\ne b, \  
 \psi_{a,b} \in [\alpha, \beta]\right\} = \mu_{\tt ST}(\alpha, \beta) p^2 +  O\bigl(p^{7/4}\bigr),
$$
where $\psi_{a,b}$ is defined similarly as $ \psi_{a,b,c}$ but with respect to the two-parametric 
family of elliptic curves  
$$
E_{a,b}: \quad Y^2 = X(X-a)(X-b) = X^3 -(a+b) X^2 +abX.
$$

Next, using  $p>3$ we define $\xi\in \F_p$ by the equation $3\xi = 1$ and make yet another change of variable $X\to X+\xi (a+b)$, 
and thus reduce our problem to prove that 
\begin{equation}
\label{eq:ST angle 2 par}
\#\left \{(a,b) \in\( \F_p^*\)^2:~a\ne b, \ \widetilde\psi_{a,b} \in [\alpha, \beta]\right\} = \mu_{\tt ST}(\alpha, \beta) p^2 + O(p^{7/4} ),
\end{equation}
where $ \widetilde\psi_{a,b}$ is the Frobenius angle of the curve
$$
 \widetilde E_{a,b}: \quad Y^2 =  X^3 + f(a, b)X + g(a, b), 
$$
which is now in  the Weierstrass form with the coefficients 
\begin{align*}
&  f(a, b) = -  \xi (a+b)^2 + ab, \\
&  g(a, b) = (\xi^3 -  \xi^2) (a+b)^3 +  \xi ab(a+b). 
\end{align*}

For every $b \in  \F_p^*$, we consider the polynomial in $\F_p[T]$: 
$$
\Delta_b(T)   = 4 \bigl(\xi (T+b)^2 - bT\bigr)^3 -  27 \bigl((\xi^3-\xi^2) (T+b)^3 +  \xi bT(T+b)\bigr)^2, 
$$
which, due to $3\xi = 1$, is in fact of degree $4$ with the non-vanishing leading 
coefficient $b^2$.  
Now, using~\cite[Theorem~2.1]{dlBSSV},
(in the very special case of $\cG = \F_p^*$), which in turn 
is based on a result of Michel~\cite[Proposition~1.1]{Mich} 
we derive that for every $b \in \F_p^*$ we have 
$$
\#\left \{a\in \F_p^*:~\Delta_b(a) \ne 0, \ \widetilde\psi_{a,b} \in [\alpha, \beta]\right\} = \mu_{\tt ST}(\alpha, \beta) p + O(p^{3/4} ),
$$
which implies~\eqref{eq:ST angle 2 par} and concludes the proof. \end{proof} 

\begin{remark}
{\rm 
 It is possible that the approach of Kaplan and  Petrow~\cite{KaPe}, see in particular~\cite[Proposition~1]{KaPe}, can provide an alternative approach to the proof of Lemma~\ref{lem:S-T abc} and extend it to arbitrary finite fields, however with a weaker  and less explicit error term just of the form $o(q^3)$.    
 }
\end{remark}

\section{Proofs of the main results}
\label{sec:proofs}

\subsection{Proof of Theorem~\ref{thm:Leaves}}
\label{sec:proof1}

Using the inclusion-exclusion formula, we write
$$
\#\(\bigcup_{i=1}^n\(\cQ_q+a_i\) \) = \sum_{j =1}^n (-1)^{j-1} 
\sum_{\substack{\cJ \subseteq\{1, \ldots, n\}\\\# \cJ =j}}
\#\(\bigcap_{i \in \cI} \(\cQ_q+a_i\)\) .
$$
We now recall Lemma~\ref{lem:ltersect Qa} and derive
$$
\left|\#\(\bigcup_{i=1}^n\(\cQ_q+a_i\) \) -  q \sum_{j =1}^n (-1)^{j +1} \binom{n}{j} 2^{-j}  \right|
< \frac{1}{2}q^{1/2}\sum_{j=1}^n   \binom{n}{j} j, 
$$
which further becomes 
$$
\left|\#\(\bigcup_{i=1}^n\(\cQ_q+a_i\) \) -  q (1-   2^{-n})  \right|
< \frac{1}{2}n 2^{n-1} q^{1/2} = n 2^{n-2} q^{1/2}. 
$$
So, we obtain 
$$
\#\Bigl(\bigcup_{i=1}^n(\cQ_q+a_i) \Bigr)  < q (1-   2^{-n}) + n 2^{n-2} q^{1/2}. 
$$
Then, if $q (1-   2^{-n}) + n 2^{n-2} q^{1/2} < q$, we equivalently have 
$$
n^2 2^{4n-4} < q. 
$$
The desired result now follows.

\subsection{Proof of Theorem~\ref{thm:Leaves2}}

Applying the same arguments as in the proof of the theorem in~\cite{GS} 
and using the Weil bound  (see~\cite[Theorem~11.23]{IwKow} or~\cite[Theorem~5.41]{LiNi}) instead of Burgess' bound as in~\cite{GS} on character sums, 
we obtain that 
when $q \equiv 3 \pmod{4}$, if $q > n^2 2^{2n-2}$, 
then for any $n$ pairwise distinct elements $a_1, \ldots, a_n \in \F_q$,  
there exists an element $u \in \F_q \setminus \{a_1, \ldots, a_n\}$ such that $a_i - u \in \cQ_q$ for any $i=1, \ldots, n$. 
This in fact refers to the Sch{\"u}tte property of the Paley graph $\cP_q$. 

Then, due to $q \equiv 3 \pmod{4}$, $-1$ is not a square, and so $u - a_i \not\in \cQ_q$ for any $i=1, \ldots, n$. 
That is, $u \not\in \bigcup_{i=1}^{n} (\cQ_q + a_i)$. 
So, $u$ is a leaf of the  directed graph $\cG(f_1, \ldots, f_n)$ of the polynomials $f_i(X)= X^2 + a_i$, $i =1, \ldots, n$. 
This completes the proof.

\subsection{Proof of Theorem~\ref{thm:NoLeaves}}
\label{sec:proof2}

We set $a_1 = 0$ and then define the subset 
$$
\cS_1 = \F_q\setminus \cQ_q.
$$ 
Clearly,  $\# \cS_1 = (q-1)/2 < q/2$.  
By Lemma~\ref{lem:Intesect}, there is an element $a_2 \in \F_q$ such that the set 
$$
\cS_{2} = \cS_1 \setminus \{u^2 +a_{2}:~u \in \F_q\}
$$
is of cardinality 
$$
\# \cS_{2} \le \frac{1}{2} \# \cS_1 < \frac{q}{2^2}. 
$$
If $a_2 = a_1$, then $\cS_{2} = \cS_{1}$, which contradicts with the above inequality. 
So, $a_2 \ne a_1$. 

Hence, using Lemma~\ref{lem:Intesect}, step by step we can construct a list of subsets $\cS_k \subseteq    \ldots \subseteq   \cS_1$ 
with pairwise distinct $a_1=0, a_2, \ldots, a_k \in \F_q$: 
$$
\cS_{i} = \cS_{i-1} \setminus \{u^2 +a_{i}:~u \in \F_q\}, i = 2, \ldots, k, 
$$
  of cardinalities
\begin{equation}
\label{eq:Ind Assump}
\# \cS_i < \frac{q}{2^{i}}, \qquad i = 1, \ldots, k. 
\end{equation}

Now, let 
$$k = \rf{\frac{\log q}{\log 2}}.
$$ 
Then, we have $2^k \ge 2^{\log q / \log 2} = q$. 
So, noticing~\eqref{eq:Ind Assump}, in this case we must have $\cS_k = \emptyset$.

Hence, by construction, there exists some positive integer $n \le \lceil \frac{\log q}{\log 2}\rceil$ 
such that 
$$
\bigcup_{i=1}^{n}(\cQ_q + a_i) = \F_q, 
$$
which concludes the proof.

\subsection{Proof of Theorem~\ref{thm:Count 2}}
\label{sec:proof count 2}
Clearly, $a$ and $b$ are not leaves of  $\cG(X^2 +a,X^2 +b)$.
Furthermore, an element $u \in \F_q \setminus\{a,b\} $ is a leaf if and only if  
 $$
 \chi(u-a) = -1 \mand \chi(u-b) = -1, 
 $$
 which is equivalent to 
$$
  \frac{1}{4} \(1- \chi(u-a)  \) \(1-\chi(u-b)  \)  =
  \begin{cases} 1, &  u \text{ is  a leaf,}\\
0,  & \text{otherwise.}
\end{cases} 
$$
Hence, we have  
\begin{equation}
\begin{split}
\label{eq:L2SSS}
 L_q(a,b) & =   
  \frac{1}{4}  \sum_{u \in \F_q \setminus\{a,b\} } \(1- \chi(u-a)  \) \(1-\chi(u-b)  \)   \\
  & =   
 \frac{1}{4}\(q-2 +S_{a,b} - S_a - S_b\), 
\end{split}
\end{equation}
where, using the multiplicativity of $\chi$, and the fact that $\chi(0) = 0$, 
\begin{align*}
& S_{a,b} =  \sum_{u \in \F_q \setminus\{a,b\} }  \chi\((u-a) (u-b)\) =   \sum_{u \in \F_q}  \chi\((u-a) (u-b)\),\\
& S_{a} =  \sum_{u \in \F_q \setminus\{a,b\} }  \chi(u-a)=   \sum_{u \in \F_q \setminus\{b\} }  \chi(u-a)
=  \sum_{u \in \F_q  }  \chi(u-a) - \chi(b-a), \\
& S_{b} =  \sum_{u \in \F_q \setminus\{a,b\} }  \chi(u-b)=   \sum_{u \in \F_q \setminus\{a\} }  \chi(u-b)
=  \sum_{u \in \F_q  }  \chi(u-b) - \chi(a-b) . 
\end{align*}

By~\cite[Theorem~5.48]{LiNi}, since $a\ne b$, we have
\begin{equation}
\label{eq:Sab}
S_{a,b} =  - 1. 
\end{equation}
Furthermore, since $\sum_{u\in \F_q} \chi(u) = 0$ (see~\cite[Equation~(5.37)]{LiNi}), we have
\begin{equation}
\begin{split}
\label{eq:Sa}
 S_{a}  &=  \sum_{u \in \F_q  }  \chi(u-a) - \chi(b-a) \\
 & = \sum_{u \in \F_q  }  \chi(u) - \chi(b-a) =  - \chi(b-a), 
\end{split}
\end{equation}
and similarly 
\begin{equation}
\begin{split}
\label{eq:Sb}
 S_{b} 
& =  \sum_{u \in \F_q  }  \chi(u-b) - \chi(a-b)\\
&  = \sum_{u \in \F_q  }  \chi(u) - \chi(a-b) =  - \chi(a-b).
\end{split}
\end{equation}
Substituting~\eqref{eq:Sab},  \eqref{eq:Sa} and~\eqref{eq:Sb} into~\eqref{eq:L2SSS}, we derive 
\begin{align*}
L_q(a,b) &  =   \frac{1}{4}\(q - 3 + \chi(b-a)  + \chi(a-b)\)\\
&  =   \frac{1}{4}\(q - 3 +   \(\chi(-1)+1\) \chi(a-b)  \).
 \end{align*}
 Since
  $$
\chi(-1) = (-1)^{(q-1)/2}  =
  \begin{cases} 1, &   \text{ if } q \equiv 1 \pmod 4, \\
-1,  &    \text{ if } q \equiv 3 \pmod 4, 
\end{cases} 
 $$
 we derive the desired result.

\subsection{Proof of Theorem~\ref{thm:Count n}}
\label{sec:proof count n}

The case $n=1$ is trivial.
Moreover, when $n=q$, that is $\{a_1, \ldots, a_n\} = \F_q$, 
clearly there is no leaf, and so the desired result holds in this case. 
 In the sequel, we assume $2 \le n \le q-1$. 

Clearly, $a_1, \ldots, a_n$ are all not leaves of  the graph $\cG(X^2 + a_1, \ldots, X^2 + a_n)$. 
As in the proof of Theorem~\ref{thm:Count 2},  for $u \in \F_q \setminus\{a_1, \ldots ,a_n\} $ we have 
$$
2^{-n} \prod_{i=1}^n \(1- \chi(u-a_i)\)    =
  \begin{cases} 1, &  u \text{ is  a leaf,}\\
0,  & \text{otherwise.}
\end{cases} 
$$
 Hence, we have 
\begin{align*}
 L_q(a_1, \ldots ,a_n) & =   
2^{-n}   \sum_{u \in \F_q \setminus\{a_1, \ldots ,a_n\} }  \prod_{i=1}^n \(1- \chi(u-a_i)\)  \\
  & =    2^{-n}  \sum_{\cJ\subseteq \{1, \ldots, n\}} (-1)^{\# \cJ}  S_\cJ, 
\end{align*}
where
$$
S_{\cJ} =   \sum_{u \in \F_q \setminus\{a_1, \ldots ,a_n\} }   \prod_{i\in \cJ}  \chi(u-a_i) 
=   \sum_{u \in \F_q \setminus\{a_1, \ldots ,a_n\} }  \chi \left( \prod_{i\in \cJ} (u-a_i) \right).
$$

For the case when $\cJ = \emptyset$, the product is equal to one for every $u$, hence the   contribution from this case is
\begin{equation}
\label{eq:S0}
S_{ \emptyset} =   q-n.
\end{equation}
Thus, we obtain  
\begin{equation}
\label{eq:LsSSS}
\left|  L_q(a_1, \ldots ,a_n) - 2^{-n}q\right|  \le 
   2^{-n} n + 2^{-n}\sum_{\substack{\cJ\subseteq \{1, \ldots, n\}\\ \cJ \ne \emptyset}} \left| S_{ \cJ} \right|.
\end{equation}

For any non-empty subset  $\cJ \subseteq \{1, \ldots, n\}$, adding back the contribution from 
$u \in \{a_1, \ldots ,a_n\}$ which is at most $n - \#\cJ$, and using the Weil bound for sums 
of muliplicative characters,  see~\cite[Theorem~5.41]{LiNi}, we obtain
\begin{equation}
\label{eq:SJ}
 \left| S_{ \cJ} \right|\le \(\#\cJ-1\) q^{1/2} + n - \#\cJ. 
\end{equation}

Substituting~\eqref{eq:SJ} into~\eqref{eq:LsSSS}, and collecting the subsets $\cJ$ 
of the same cardinality $j = \#\cJ$, we obtain
\begin{align*}
\left|  L_q(a_1, \ldots ,a_n) - 2^{-n}q\right|  & \le  2^{-n} n + 
2^{-n}  \sum_{j=1}^n \binom{n}{j}  \( (j-1) q^{1/2} + n - j\)\\
& = \frac{1}{2}n   + 2^{-n}  q^{1/2}  \sum_{j=1}^n \binom{n}{j} (j-1) .
\end{align*}
Since 
$$
 \sum_{j=1}^n \binom{n}{j} (j-1)   =  \sum_{j=1}^n \binom{n}{j} j - 2^n +1 
 = n2^{n-1} - 2^n + 1,
 $$
 we derive 
\begin{equation} \label{eq:Lqa1}
\begin{split}
\left|  L_q(a_1, \ldots ,a_n) - 2^{-n}q\right|  
& \le \frac{1}{2} n   + 2^{-n}  q^{1/2}(n2^{n-1} - 2^n + 1)\\
& = \frac{1}{2} n   + \(\frac{n}{2} + \frac{1}{2^n} -  1\)q^{1/2} \\
& <  \frac{2}{3} n q^{1/2}, 
\end{split}
\end{equation}
where the last inequality is trivial if $q \ge 9$, 
and for $q=3, 5, 7$, one can verify it by direct computation and noticing $2 \le n \le q-1$.  
This completes the proof of the first part of Theorem~\ref{thm:Count n}.

Now, to replace the factor $2/3$ by $1/2$ in the last inequality in~\eqref{eq:Lqa1}, it is equivalent to ensure that 
$$
 \frac{1}{2} n   + \(\frac{1}{2^n} -  1\)q^{1/2} < 0,
$$
which equivalently becomes
\begin{equation} \label{eq:n24q-}
n^2 < 4q  -  \(\frac{1}{2^{n-3}} - \frac{1}{2^{2n-2}}\)q.
\end{equation}
Since $q$ is an odd prime power and $2 \le n \le q-1$, the inequality~\eqref{eq:n24q-} holds when $n=2, 3, 4$. 
So, in the following we assume $5 \le n \le q-1$ (automatically, $q \ge 7$). 

 Since $n \ge 5$, we have 
$$
4q  -  \(\frac{1}{2^{n-3}} - \frac{1}{2^{2n-2}}\)q  \ge 4q  -  \(\frac{1}{2^{2}} - \frac{1}{2^{8}}\)q = \frac{961}{256}q.
$$
So, if 
$$n^2 < \frac{961}{256}q
$$ 
(that is, $n < \frac{31}{16}q^{1/2}$), then the inequality~\eqref{eq:n24q-} automatically holds. 

If 
$$
n^2 \ge \frac{961}{256}q
$$ 
(that is, $n \ge \frac{31}{16}q^{1/2}$), then we deduce that 
$$
\frac{q}{2^{n-3}} \le \frac{q}{2^{\frac{31}{16}\sqrt{q}-3}} \le \frac{7}{2^{\frac{31}{16}\sqrt{7}-3}} < 2, 
$$
which, together with the fact that $n^2 \ne 4q -1$ (because $n^2 \not\equiv -1 \pmod{4}$),  
implies that in this case the inequality~\eqref{eq:n24q-} is equivalent to 
\begin{equation} \label{eq:n24q}
n^2 < 4q.
\end{equation} 

Therefore, when $n^2 < 4q$ (that is, $n < 2q^{1/2}$), the inequality~\eqref{eq:n24q-} always holds, 
and so  the factor $2/3$  in the last inequality in~\eqref{eq:Lqa1} can be replaced by 
$1/2$. 
This completes the proof.

\subsection{Proof of Theorem~\ref{thm:Count 3}}
\label{sec:proof count 3}

As in the proof of Theorem~\ref{thm:Count 2} and~\ref{thm:Count n}, 
we have the following analogue of~\eqref{eq:L2SSS}, where 
\begin{align*}
 L_q(a, b, c) & =   
\frac{1}{8}   \sum_{u \in \F_q \setminus\{a,b,c\} }  (1- \chi(u-a)) (1-\chi(u-b)) (1-\chi(u-c))  \\
& =  \frac{1}{8}(q-3) - \frac{1}{8} \(S_a + S_b + S_c\)  \\
& \qquad \qquad \qquad + \frac{1}{8}\(S_{a,b} + S_{a,c}+ S_{b,c}\) 
- \frac{1}{8}S_{a,b,c}, 
\end{align*}
where 
\begin{align*} 
S_a &= \sum_{u \in \F_q \setminus\{a,b,c\} }  \chi(u-a)  = \sum_{u \in \F_q \setminus\{b,c\}}  \chi(u-a) \\
&  = \sum_{u \in \F_q }  \chi(u-a) - \chi(b-a) - \chi(c-a) \\
&   = - \chi(b-a) - \chi(c-a), \\ 
  S_b &= \sum_{u \in \F_q \setminus\{a,b,c\} }  \chi(u-b) = - \chi(a-b) - \chi(c-b), \\
 S_c & = \sum_{u \in \F_q \setminus\{a,b,c\} }  \chi(u-c) = - \chi(a-c) - \chi(b-c), 
\end{align*}
and (noticing~\eqref{eq:Sab})
\begin{align*} 
  S_{a,b} &= \sum_{u \in \F_q \setminus\{a,b,c\} }  \chi((u-a)(u-b))  = \sum_{u \in \F_q \setminus\{c\}}  \chi((u-a)(u-b)) \\
& \quad = \sum_{u \in \F_q }  \chi((u-a)(u-b)) - \chi((c-a)(c-b)) \\
& \quad = - 1 - \chi((c-a)(c-b)), \\ 
  S_{a,c} & = \sum_{u \in \F_q \setminus\{a,b,c\} }  \chi((u-a)(u-c))  = - 1 - \chi((b-a)(b-c)), \\
 S_{b,c} & = \sum_{u \in \F_q \setminus\{a,b,c\} }  \chi((u-b)(u-c))  = - 1 - \chi((a-b)(a-c)), 
\end{align*}
and finally 
$$
 S_{a,b,c} =    \sum_{u \in \F_q \setminus\{a,b,c\} }  \chi\bigl((u-a) (u-b) (u-c)\bigr). 
$$
So, it remains to compute $S_{a,b,c}$. 

Note that for any $u \in \F_q \setminus\{a,b,c\}$, we have that 
$\chi((u-a) (u-b) (u-c))=1$ if and only if $(u-a) (u-b) (u-c)$ is a non-zero square in $\F_q$, 
which is equivalent to that $(u, \pm y) \in E_{a,b,c}(\F_q)$ for some non-zero $y \in \F_q$. 
Since $(a, 0), (b,0), (c,0) \in E_{a,b,c}(\F_q)$, we obtain that there are exactly 
$$
\frac{\#E_{a,b,c}(\F_q) - 3}{2}
$$
elements $u \in \F_q\setminus\{a,b,c\}$ such that $\chi((u-a) (u-b) (u-c))=1$. 
So, there are exactly 
$$
q-3 - \frac{\#E_{a,b,c}(\F_q) - 3}{2}
$$
elements $u \in \F_q\setminus\{a,b,c\}$ such that $\chi((u-a) (u-b) (u-c))=-1$. 
Hence, we have 
\begin{align*}
 S_{a,b,c} & = \frac{\#E_{a,b,c}(\F_q) - 3}{2} - \( q-3 - \frac{\#E_{a,b,c}(\F_q) - 3}{2} \) \\
 & = \#E_{a,b,c}(\F_q) - q = 1 - t_{a,b,c}. 
\end{align*}

Therefore, collecting the above computations we obtain the desired result.

\subsection{Proof of Corollary~\ref{cor:Count 3}}
By Theorem~\ref{thm:Count 3} and using the Hasse bound~\eqref{eq:Hasse}, we obtain  
$$
 \left| L_q(a,b,c) - \frac{q}{8} \right|  =  \frac{1}{8}|t_{a,b,c}|  + \frac{7}{8} + \frac{3}{4} + \frac{3}{8} \le \frac{1}{4}q^{1/2} + 2. 
$$
This completes the proof. 


\subsection{Proof of Theorem~\ref{thm:Count 3+}}
\label{sec:proof count 3+}
By Theorem~\ref{thm:Count 3},  we have 
$$
L_p(a,b,c)  -  \frac{p}{8} = \frac{1}{8}  t_{a,b,c}  + O(1).
$$
Hence 
$$
\Delta_p(a,b,c)   =    \frac{1}{8p^{1/2}} t_{a,b,c}  + O\bigl(p^{-1/2}\bigr)
=      \frac{1}{4} \cos \psi_{a,b,c}  + O\bigl(p^{-1/2}\bigr).  
$$ 
where $\psi_{a,b,c}\in [0, \pi]$  is defined by~\eqref{eq:ST angle}. 
Hence,  
$$
\Delta_p(a,b,c) \in [\sigma,\tau]
$$
is equivalent to 
$$
\psi_{a,b,c} \in [ \alpha, \beta] 
$$
for some $[\alpha, \beta] \subseteq [0, \pi]$ with 
$$
\alpha = \arccos(4\tau) + O\bigl(p^{-1/2}\bigr) \mand \beta =  \arccos(4\sigma) + O\bigl(p^{-1/2}\bigr).
$$
Recalling Lemma~\ref{lem:S-T abc}  we obtain the desired result.

\section{Numerical results}
\label{sec:numerical}  

\subsection{Outline}
Here  we present extensive numerical results in order to provide some insights on the distribution of leaves for functional graphs of families of quadratic polynomials. These results show highly regular behaviour of several parameters associated with our functional graphs. Hence, we hope that they 
may   lead to future conjectures and consequently to their resolutions.

\subsection{Set-up}
Recall that $q$ is some power of an odd prime $p$. 

It is clear that naively considering graphs built from all possible combinations of $n$ polynomials in their generic forms 
(that is, with all possible coefficients from $\F_q$)  leads to inefficiency and unacceptable computational delays.

In order to try large $n$ and large $q$, 
here we only consider the graphs  $\cG(X^2+a_1, \ldots, X^2 + a_n)$ over  $\F_q$
with pairwise distinct $a_1, \ldots, a_n \in \F_q$. 
Our purpose is to numerically study the distribution of leaves for these graphs.  
The most natural way is to make computations for all these graphs, which we call \textit{the brute force approach}. 
Clearly, the total number of these graphs is the binomial coefficient: 
\begin{equation}   \label{eq:brute}
\binom{q}{n}  \sim \frac{q^n}{n!}, 
\end{equation}
where the asymptotic formula holds when $n$ is fixed and $q$ tends to the infinity. 

In the sequel we present an approach, which we call \textit{the simplified approach}, to numerically study the distribution of leaves for such graphs. 
We use this approach in our computations. 

The following lemma enables us to design the simplified approach. 

\begin{lemma}  \label{lem:leaf}
Let $a_1, \ldots, a_n \in \F_q$ be pairwise distinct. 
Then, for any $u \in \F_q^*$ the graph $\cG(X^2+a_1 u^2, \ldots, X^2 + a_n u^2)$ over  $\F_q$ 
has the same number of leaves as the graph $\cG(X^2+a_1, \ldots, X^2 + a_n)$. 
\end{lemma}

\begin{proof}
First, we define a bijection, say $\varphi$, from $\cG(X^2+a_1, \ldots, X^2 + a_n)$ to $\cG(X^2+a_1 u^2, \ldots, X^2 + a_n u^2)$ 
by setting $\varphi(x) = u^2 x$ for any $x \in \F_q$. 
Then, it suffices to show that a vertex $y$ in $\cG(X^2+a_1, \ldots, X^2 + a_n)$ is not a leaf if and only if the vertex $u^2 y$ 
in $\cG(X^2+a_1 u^2, \ldots, X^2 + a_n u^2)$ is not a leaf.   

If a vertex $y$ in $\cG(X^2+a_1, \ldots, X^2 + a_n)$ is not a leaf, then there are some $x \in \F_q$ and $a_i$ such that $y = x^2 + a_i$, 
which gives $u^2 y = (ux)^2 + a_i u^2$, and so the vertex $u^2 y$ is also not a leaf in the graph $\cG(X^2+a_1 u^2, \ldots, X^2 + a_n u^2)$. 
Similarly, we can show that if $u^2y$ is not a leaf in $\cG(X^2+a_1 u^2, \ldots, X^2 + a_n u^2)$, then $y$ is also not a leaf in $\cG(X^2+a_1, \ldots, X^2 + a_n)$. 
This  completes the proof. 
\end{proof}

Now, we fix a non-square, say $\lambda$, in $\F_q$. 
For any integer $k \ge 0$, let $N_{k}$ be the number of the graphs  $\cG(X^2 + a_1, \ldots, X^2 + a_n)$ over  $\F_q$
with  pairwise distinct $a_1, \ldots, a_n \in \F_q$ and with (exactly) $k$ leaves. 
For any integer $k \ge 0$ and $a_1 \in \{0, 1, \lambda\}$, let $N_{k}^{a_1}$ be the number of the graphs  $\cG(X^2 + a_1, \ldots, X^2 + a_n)$ over  $\F_q$
with  pairwise distinct $a_1, \ldots, a_n \in \F_q$, $a_2\cdots a_n \ne 0$   and with $k$ leaves.

\begin{lemma}  \label{lem:num}
For any integer $k \ge 0$, we have 
$$
N_k = N_{k}^{0} + \frac{ (q-1)(N_{k}^{1} + N_{k}^{\lambda})}{2n}. 
$$
\end{lemma}

\begin{proof}
First, clearly we have 
$$
N_k = N_{k}^{0} + N_k^*, 
$$
where $N_k^*$ is the number of the graphs  $\cG(X^2 + a_1, \ldots, X^2 + a_n)$ over  $\F_q$
with  pairwise distinct $a_1, \ldots, a_n \in \F_q^*$  and with  $k$ leaves. 

By Lemma~\ref{lem:leaf}, for each graph  $\cG(X^2 + a_1, \ldots, X^2 + a_n)$ contributing to $N_{k}^{1}$ or  $N_{k}^{\lambda}$, 
the graph $\cG(X^2 + u^2a_1, \ldots, X^2 + u^2a_n)$ contributes to $N_{k}^*$ for any $u \in \F_q^*$. 
This leads to the quantity $\frac{q-1}{2}(N_{k}^{1} + N_{k}^{\lambda})$.

In addition, each graph $\cG(X^2 + a_1, \ldots, X^2 + a_n)$ contributing to $N_{k}^*$ is counted $n$ times in $\frac{q-1}{2}(N_{k}^{1} + N_{k}^{\lambda})$. 
Indeed, each graph  $\cG(X^2 + a_1, \ldots, X^2 + a_n)$ is essentially generated by the set $\{a_1, \ldots, a_n\}$.  
For each $i$ with $1\le i \le n$,  $\{a_1, \ldots, a_n\}=a_i \cdot \{a_i^{-1} a_1, \ldots, a_i^{-1} a_n\}$ 
corresponding to a graph counted in $\frac{q-1}{2}N_{k}^{1}$ if $a_i$ is a square, 
and otherwise, $\{a_1, \ldots, a_n\}=\lambda^{-1} a_i \cdot \{\lambda a_i^{-1} a_1, \ldots, \lambda a_i^{-1} a_n\}$ 
corresponding to a graph counted in $\frac{q-1}{2}N_{k}^{\lambda}$. 

Therefore, we obtain 
$$
N_k = N_{k}^{0} + N_k^* = N_{k}^{0} + \frac{ (q-1)(N_{k}^{1} + N_{k}^{\lambda})}{2n}. 
$$
\end{proof}

By Lemma~\ref{lem:num}, we directly get the simplified approach: 
for our purpose we only need to make computations 
for the graphs $\cG(X^2 + a_1, \ldots, X^2 + a_n)$ over  $\F_q$
with $a_1 \in \{0, 1, \lambda\}$, pairwise distinct $a_1, \ldots, a_n \in \F_q$, and $a_2\cdots a_n \ne 0$. 
Clearly, the number of these graphs is equal to 
the sum of three binomial coefficients: 
\begin{equation}  \label{eq:simple}
\binom{q-1}{n-1} + \binom{q-2}{n-1} + \binom{q-3}{n-1} \sim \frac{3q^{n-1}}{(n-1)!}, 
\end{equation}
where the asymptotic formula holds when $n$ is fixed and $q$ tends to the infinity. 

Comparing~\eqref{eq:simple} with~\eqref{eq:brute}, we can see that 
the simplified approach is better than the brute force approach when $q$ is large. 

In the sequel, for simplicity we only make computations for  the graphs $\cG(X^2 + a_1, \ldots, X^2 + a_n)$ over  $\F_p$
with $a_1 \in \{0, 1, \lambda\}$, pairwise distinct $a_1, \ldots, a_n \in \F_p$, and $a_2\cdots a_n \ne 0$. 
Then, using Lemma~\ref{lem:num} we obtain the distribution of leaves for the graphs  $\cG(X^2+a_1, \ldots, X^2 + a_n)$ over  $\F_p$
with pairwise distinct $a_1, \ldots, a_n \in \F_p$.  

In addition, we always identify $\F_p$ as $\{0, 1, \ldots, p-1\}$, 
and in our computations we choose $\lambda$ to be the smallest non-square in $\F_p^*$.

\subsection{Algorithms} 
Recall that $\cQ_p$ is the set of squares in $\F_p$. 
Formally, a vertex $b$ in $\cG(X^2+a_1,  \ldots, X^2 + a_n)$ is a leaf if and only if 
there is no element $u \in \F_p$ and $a_i$ such that $u^2 + a_i = b$, that is $b - a_i \in \cQ_p$. 

We first present Algorithm~\ref{alg:leaf1} for counting leaves in the graph $\cG(X^2+a_1,  \ldots, X^2 + a_n)$ over $\F_p$. 
In Algorithm~\ref{alg:leaf1}, 
since we have identified $\F_p$ as $\{0, 1, \ldots, p-1\}$, 
the number of 0's in each binary string $B_i$ is exactly the number of leaves in the graph $\cG(X^2 + a_i)$ generated by the polynomial $X^2 + a_i$, 
 and thus, the number of 0's in the binary string $C$ 
is exactly the number of leaves in the graph $\cG(X^2+a_1,  \ldots, X^2 + a_n)$. 

\begin{algorithm}[H]
\begin{algorithmic}[1]
	\REQUIRE pairwise distinct $a_1, \ldots, a_n \in \F_p$
	\ENSURE the number of leaves in $\cG(X^2+a_1,  \ldots, X^2 + a_n)$
	\STATE  Create a zero array $A$ of length $p$.
	\STATE  Put $A[j] = j^2 \in \F_p, j=0, \ldots, p-1$.
	\FOR{$i=1, \ldots, n$}
		\STATE  Create a zero binary string $B_i$ of length $p$.
		\STATE Put the $(A[j]+a_i)$-th entry of $B_i$ to be 1, $j=0, \ldots, p-1$.
	\ENDFOR
    \STATE  Compute the bitwise OR, say $C$, of $B_1, \ldots, B_n$.
	\STATE  Output the number of 0's in $C$.  
\end{algorithmic}
\caption{Counting leaves in $\cG(X^2+a_1,  \ldots, X^2 + a_n)$}
\label{alg:leaf1}
\end{algorithm}

In Algorithm~\ref{alg:leaf1}, although it needs much memory to store those arrays and strings, 
it is very suitable for implementation and fast computation in the C++ programming.  So, we use it in our computations.

For completeness, we now present another algorithm in Algorithm~\ref{alg:leaf2} to counting leaves in the graph  
 $\cG(X^2 + a_1,  \ldots, X^2 + a_n)$ over  $\F_p$, which needs less memory than the one in Algorithm~\ref{alg:leaf1}. 
In Algorithm~\ref{alg:leaf2}, the final value of the variable $N$ is the number of leaves. 

\begin{algorithm}[H]
\begin{algorithmic}[1]
	\REQUIRE pairwise distinct $a_1, \ldots, a_n \in \F_p$
	\ENSURE the number of leaves in $\cG(X^2+a_1,  \ldots, X^2 + a_n)$
	\STATE Define an integer variable $N$, and put $N=0$. 
	\FOR{$b=0, 1, \ldots, p-1$}
	\FOR{$i=1, \ldots, n$}
		\IF {$b - a_i \in \cQ_p$}
			\STATE  Go to Step 2 with next $b$. 
		\ELSE
			\STATE  $i \leftarrow i+1$.
		\ENDIF
	\ENDFOR
	   \IF {$i=n+1$}
			\STATE  $N \leftarrow N+1$. 
		\ENDIF
	\ENDFOR
	\STATE Output $N$. 
\end{algorithmic}
\caption{Counting leaves in $\cG(X^2+a_1,  \ldots, X^2 + a_n)$}
\label{alg:leaf2}
\end{algorithm}

\subsection{Number of leaves} 

We first present the upper and lower bounds for the number of leaves of the graphs $\cG(X^2+a_1, \ldots, X^2 + a_n)$ over  $\F_p$ 
with pairwise distinct $a_1, \ldots, a_n \in \F_p$    
 for all primes $p$ with $p \leq 1000$. In Figure~\ref{fig:upperlowerbounds}, for the sake of clarity, we limit the presentation to the cases  $n=3, 4, 5$. It first confirms the result in Theorem~\ref{thm:Leaves} numerically that leafless graphs do not exist (that is, the lower bound is greater than zero) for a fixed $n$ as soon as $p$ is sufficiently large. For example, there is no leafless graphs for $n=3$ when $p \ge19$, for $n=4$ when $p \ge 107$. 

In addition, it shows some (non-strictly) monotonic behaviour where both bounds of the number of leaves increase with respect to $p$.

\begin{figure}[H]
\begin{center}
\includegraphics[scale=0.95]{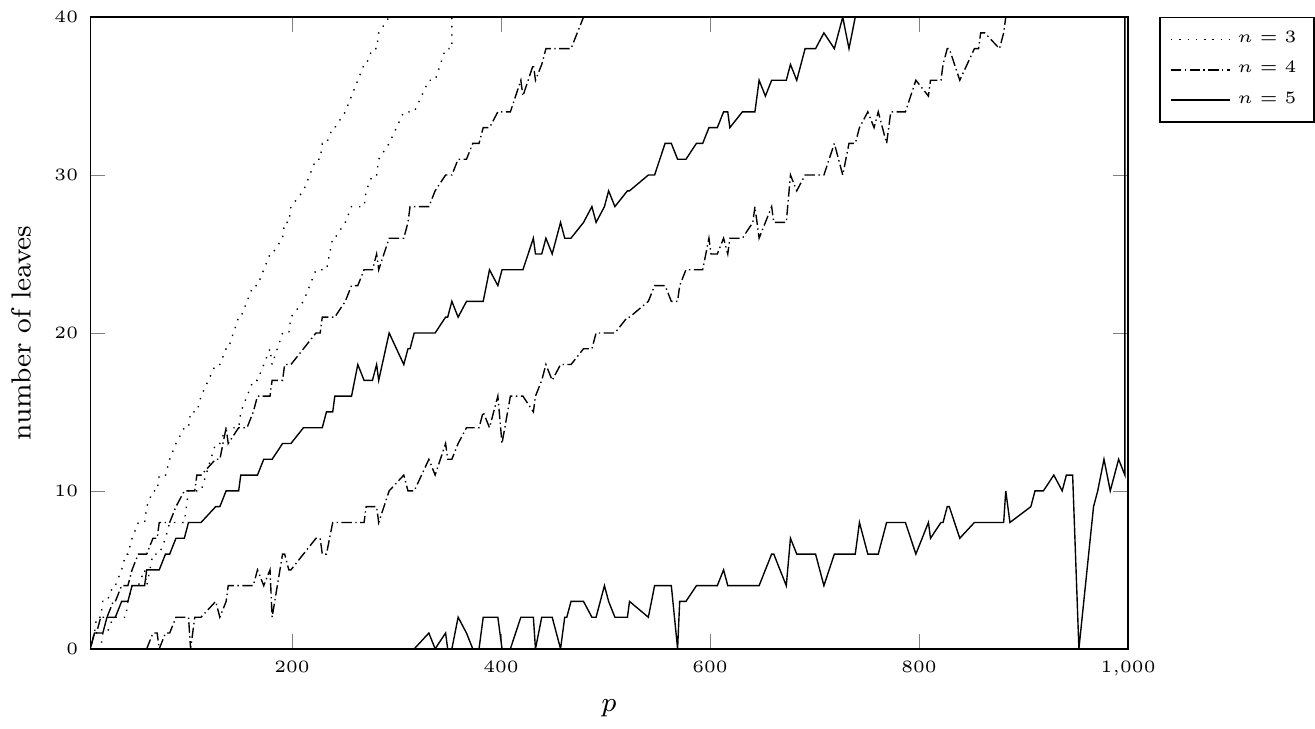}
\end{center}
\caption{Upper and lower bounds for the number of leaves when $n=3, 4, 5$.}
\label{fig:upperlowerbounds}
\end{figure}

\subsection{Number of graphs with particular number of leaves} 
Next, we consider the number of graphs with particular number of leaves for each value of $p$. We consider the cases $n=3, 4, 5, 6$ individually, in Figures~\ref{fig:leave3poly}--\ref{fig:leave6poly} respectively. 
To facilitate the presentation in the same figure, we scale the number using the natural logarithm, which is called the \textit{log-number of graphs}. 
That is, the log-number of graphs $\cG(X^2+a_1, \ldots, X^2 + a_n)$ with  $k$ leaves 
is equal to the natural logarithm of the number of graphs $\cG(X^2+a_1, \ldots, X^2 + a_n)$ with  $k$ leaves 
when $a_1, \ldots, a_n$ run over  pairwise distinct elements of $\F_p$.  

In Figures~\ref{fig:leave3poly}--\ref{fig:leave6poly}, each broken line represents the log-number of graphs with a particular number $k$ of leaves 
($1 \le k \le 39$). 
One can see that with respect to $p$, the number of such graphs first increases and then reach a peak and then decreases. 
In addition, one can also see that when $k$ increases, the first prime $p$ at which the corresponding plot of log-numbers originates also increases.

\begin{figure}[H]
\begin{center}
\includegraphics{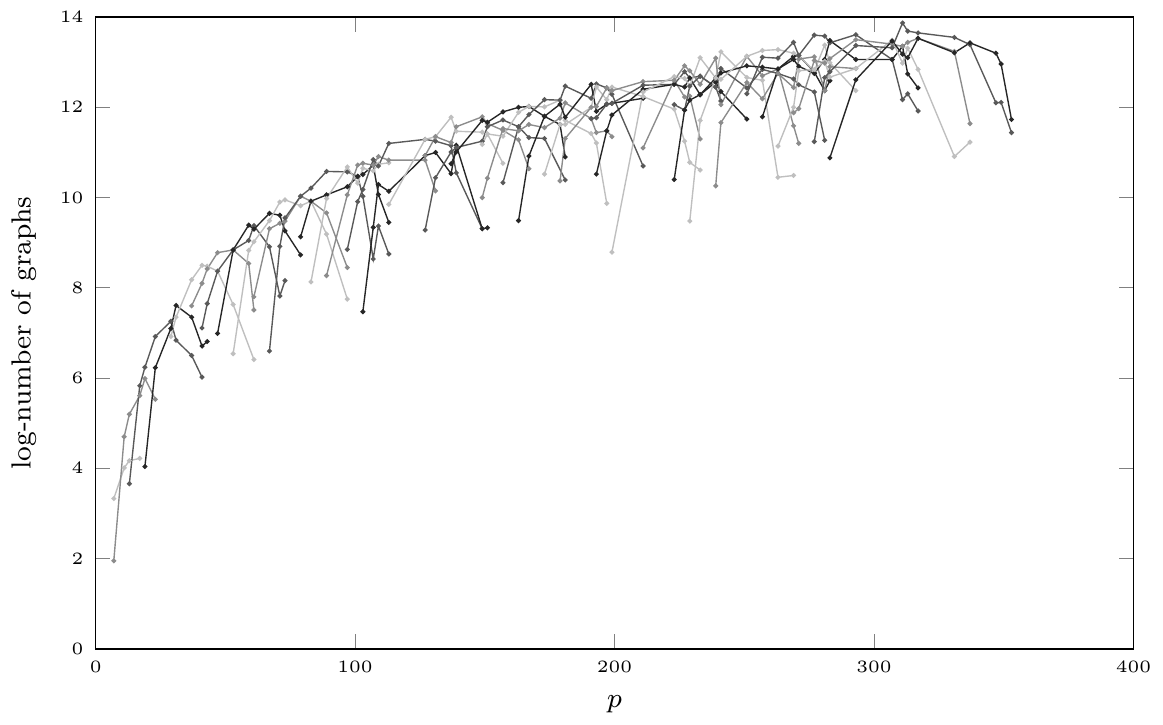}
\end{center}
\caption{Log-number of graphs with particular number of leaves when $n=3$}
\label{fig:leave3poly}
\end{figure}

\begin{figure}[H]
\begin{center}
\includegraphics{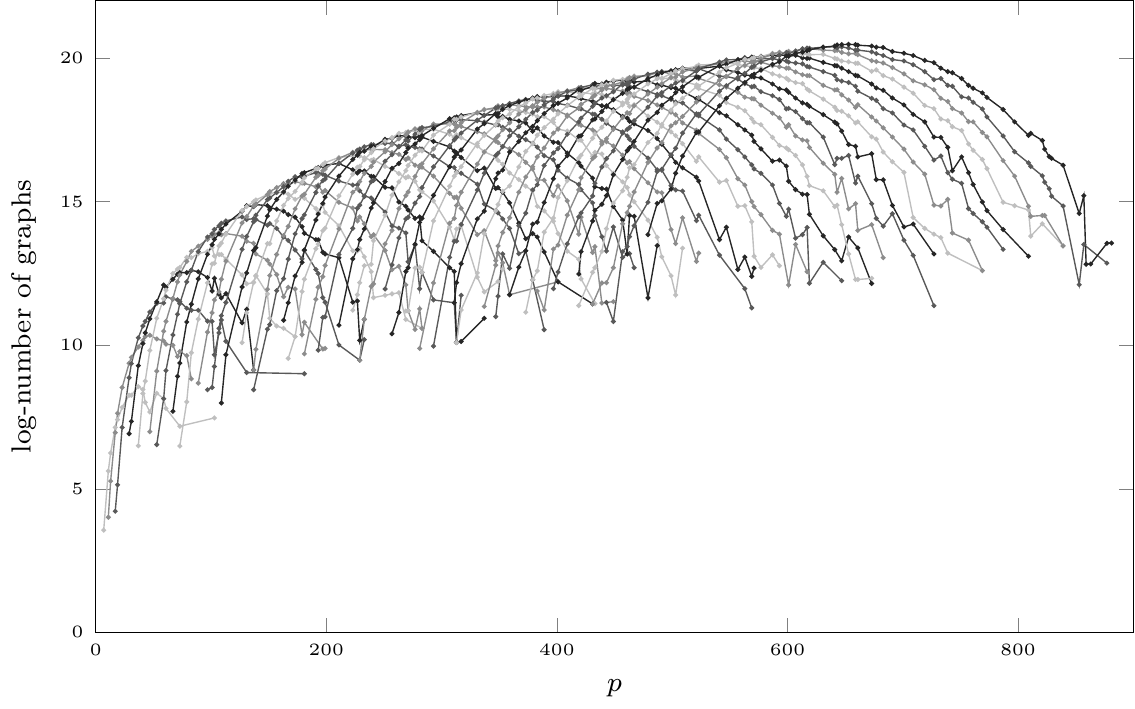}
\end{center}
\caption{Log-number of graphs with particular number of leaves when $n=4$}
\label{fig:leave4poly}
\end{figure}

\begin{figure}[H]
\begin{center}
\includegraphics{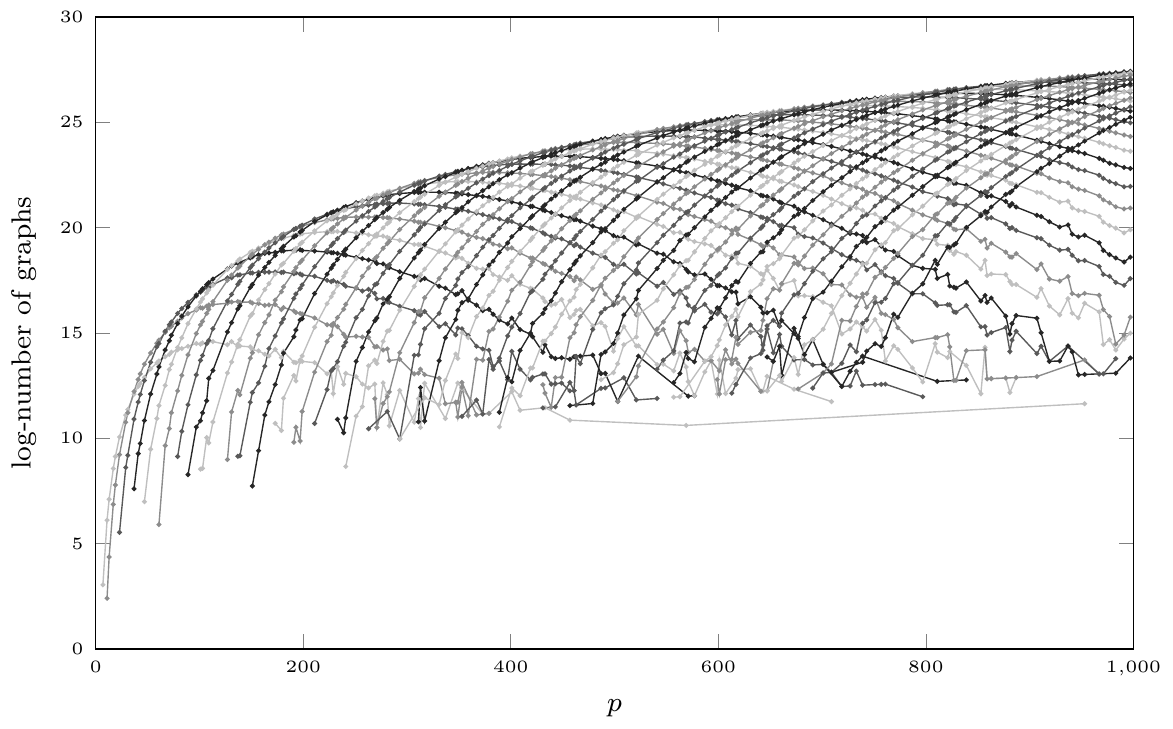}
\end{center}
\caption{Log-number of graphs with particular number of leaves when $n=5$}
\label{fig:leave5poly}
\end{figure}

\begin{figure}[H]
\begin{center}
\includegraphics{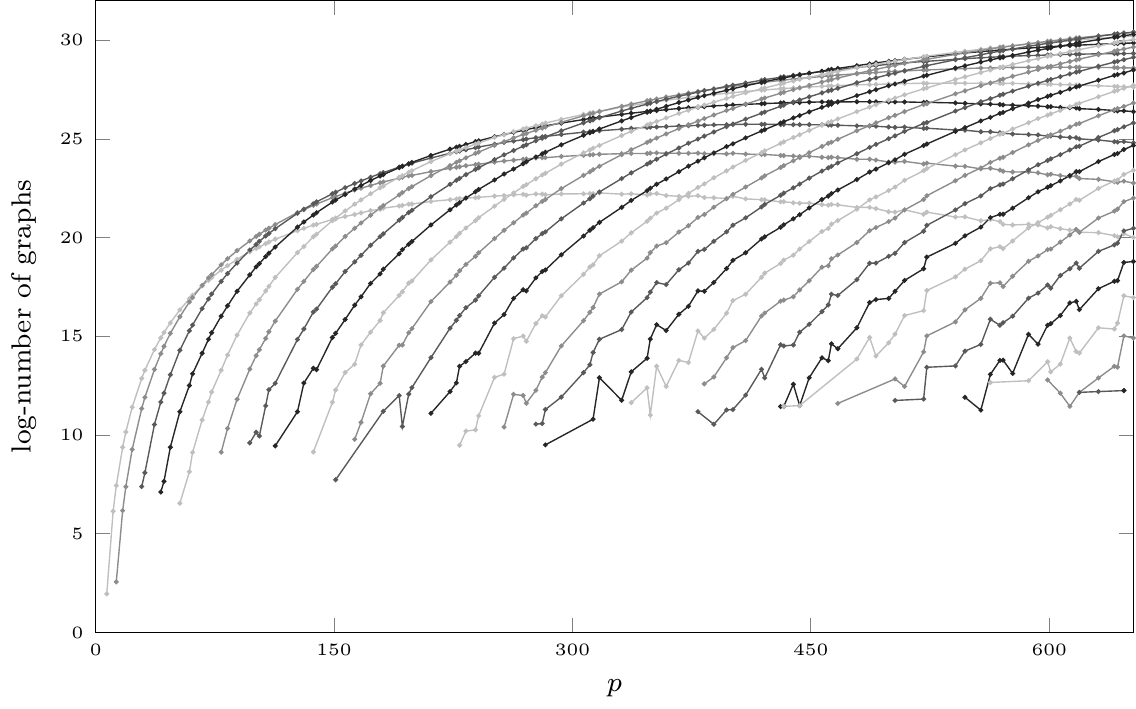}
\end{center}
\caption{Log-number of graphs with particular number of leaves when $n=6$}
\label{fig:leave6poly}
\end{figure}

\subsection{Proportion of graphs with particular number of leaves} 

Finally, we present numerical results on the proportion of graphs (as a percentage) with particular number of leaves for all the graphs 
$\cG(X^2+a_1, \ldots, X^2 + a_n)$, with pairwise distinct  $a_1, \ldots, a_n \in \F_p$, for the first 100 primes (that is, until $p=541$) up to $n=6$. 
The numerical results reflect the trend of these proportions. 
 In the counting computations, for simplicity we put the graphs with at least 40 leaves to the same set. 
 
 We present the cases $n=3, 4, 5, 6$ individually, in Figures~\ref{fig:prop3}--\ref{fig:prop6} respectively. 
In Figures~\ref{fig:prop3}--\ref{fig:prop6}, each stacked bar chart corresponding to a prime $p$ represents 
the proportions of graphs with particular numbers $k$ of leaves, 
and moreover, from the bottom to the top the number $k$ increases.  

For example, in Figure~\ref{fig:prop3}, corresponding to $n=3$, for the prime $p=71$, the graphs with 6 leaves occupy 4.35\% (the corresponding sub-bar is at the bottom), 
those with 7 leaves occupy 26.09\%, 
those with 8 leaves occupy 34.78\%, those with 9 leaves occupy 21.74\%, and those with 10 leaves occupy 13.04\% (the corresponding sub-bar is at the top).

In addition, since we only use five different shades of grey to draw the bar charts, 
it can happen that in one stacked bar chart some sub-bars have the same appearance, 
but they definitely correspond to different values of $k$. 

In Figure~\ref{fig:prop3},  
for primes $p \ge 359$, all the graphs have at least 40 leaves 
(note that we put the graphs with at least 40 leaves to the same set). 
In addition, for primes $71 \le p < 359$, the minimal number of leaves of the graphs we consider for each prime $p$ 
is presented in the column `leaves' in Table~\ref{tab:n=3}, and it also corresponds to the bottom sub-bar of the stacked bar chart of $p$ in Figure~\ref{fig:prop3}. 
From Table~\ref{tab:n=3} one can see that this minimal number roughly increases when $p$ grows. 
This also has been suggested by Figure~\ref{fig:upperlowerbounds}. 

In Figures~\ref{fig:prop4}--\ref{fig:prop6}, the nearby sub-bars having the same
shade of grey correspond to the same $k$ (number of leaves), 
and they form a strip in the figures. 
In addition, in each figure different strips correspond to different $k$. 
It is clear from Figures~\ref{fig:prop4}--\ref{fig:prop6} that the proportion of graphs with a fixed number $k$ of leaves 
first increases and then decreases. 
 This  shows again a similar distribution of graphs with  $k$ leaves for any $k$ as in Figures~\ref{fig:leave3poly}--\ref{fig:leave6poly}.  
This also suggests that an approximation of the number of   graphs with  $k$ leaves for  given $p$ and $n$ might be predictable.

\begin{table}  
\begin{center}
\begin{tabular}{|c|c|c|c|c|c|c|c|c|c|c|}  
\hline
$p$ & leaves & & $p$ & leaves & & $p$ & leaves & & $p$ & leaves \\
\cline{1-2}  \cline{4-5}  \cline{7-8} \cline{10-11}
71 & 6 & & 137 & 14  & & 199 & 21 & & 277 & 30 \\
\cline{1-2}  \cline{4-5}  \cline{7-8} \cline{10-11}
73  & 6  & & 139  & 14   & & 211   & 22  & & 281 & 30 \\
\cline{1-2}  \cline{4-5}  \cline{7-8} \cline{10-11}
79  & 7  & & 149  &  14  & & 223   & 24  & & 283  & 31 \\
\cline{1-2}  \cline{4-5}  \cline{7-8} \cline{10-11}
83  & 8  & & 151  &  15  & &  227  & 24  & &  293 & 32  \\
\cline{1-2}  \cline{4-5}  \cline{7-8} \cline{10-11}
89  & 8  & & 157  & 16   & &  229  & 24  & &  307 & 34 \\
\cline{1-2}  \cline{4-5}  \cline{7-8} \cline{10-11}
97  & 8  & & 163  & 17   & &  233  & 24  & & 311  & 34 \\
\cline{1-2}  \cline{4-5}  \cline{7-8} \cline{10-11}
101  & 10  & & 167  & 17  & & 239   &  26 & & 313  & 34 \\
 \cline{1-2}  \cline{4-5}  \cline{7-8} \cline{10-11}
103  & 10  & & 173  & 18   & &  241  & 26  & & 317  & 34 \\
\cline{1-2}  \cline{4-5}  \cline{7-8} \cline{10-11}
107  & 10  & &  179 &  19  & &  251  & 27  & & 331  &  36 \\
\cline{1-2}  \cline{4-5}  \cline{7-8} \cline{10-11}
109  & 10  & & 181  & 18   & & 257   & 28  & & 337  & 36 \\
 \cline{1-2}  \cline{4-5}  \cline{7-8} \cline{10-11}
113  & 10  & & 191  &  20  & &  263  & 28  & & 347  & 38  \\
\cline{1-2}  \cline{4-5}  \cline{7-8} \cline{10-11}
127  & 13  & & 193  & 20   & & 269   & 28  & & 349  & 38  \\
\cline{1-2}  \cline{4-5}  \cline{7-8} \cline{10-11}
131  & 13  & & 197  &  20  & &  271  & 29  & &  353 & 38 \\
\hline
\end{tabular}
\vspace{5mm}
\caption{Minimal number of leaves when $n=3$}
\label{tab:n=3}
\end{center}
\end{table}

\begin{figure}[H]
\begin{center}
\includegraphics{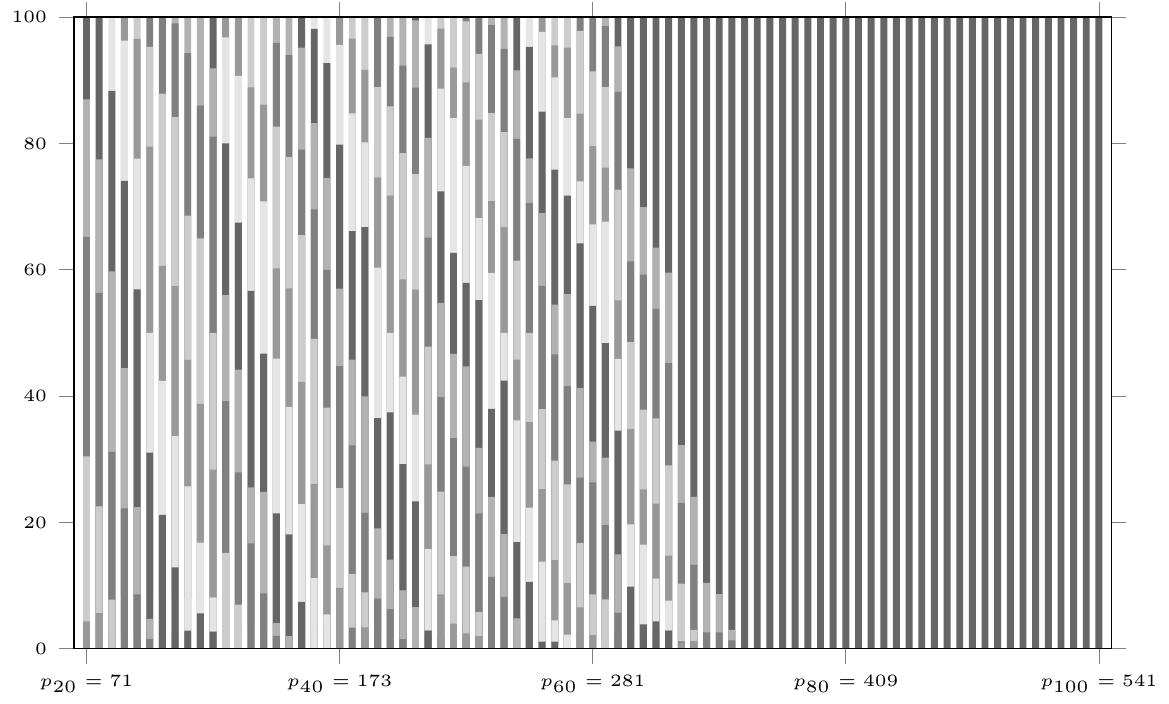}
\end{center}
\caption{Proportion of number of leaves when $n=3$}
\label{fig:prop3}
\end{figure}

\begin{figure}[H]
\begin{center}
\includegraphics{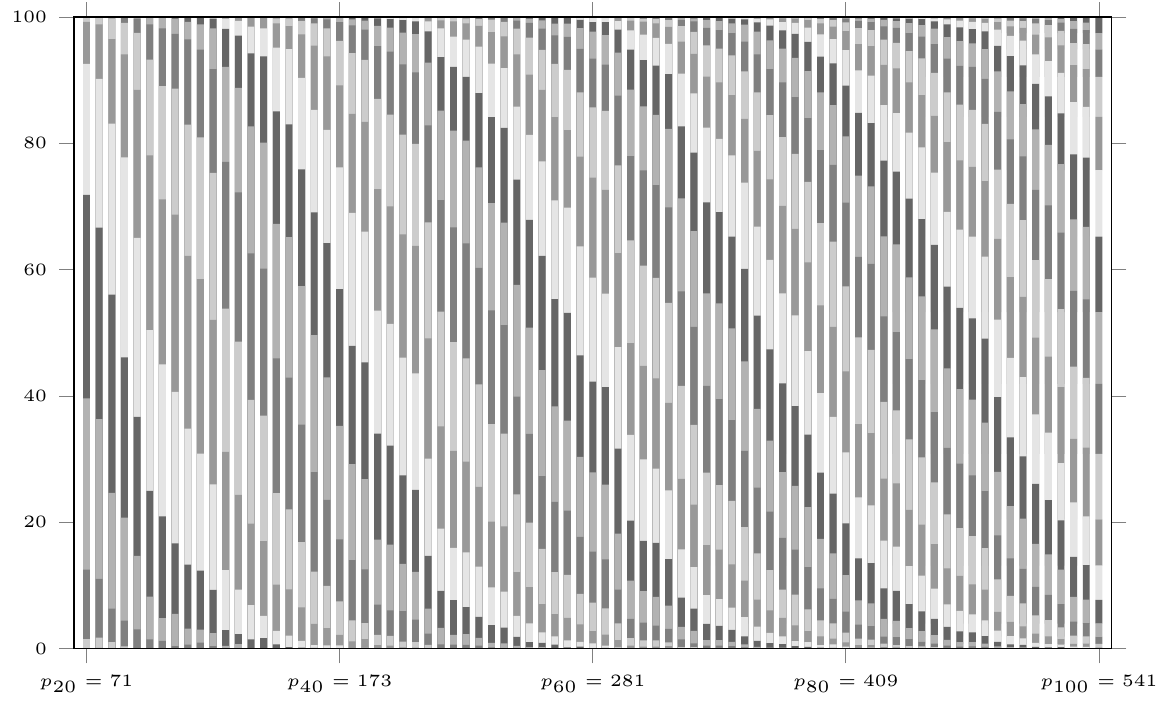}
\end{center}
\caption{Proportion of number of leaves when $n=4$}
\label{fig:prop4}
\end{figure}

\begin{figure}[H]
\begin{center}
\includegraphics[scale=1.0]{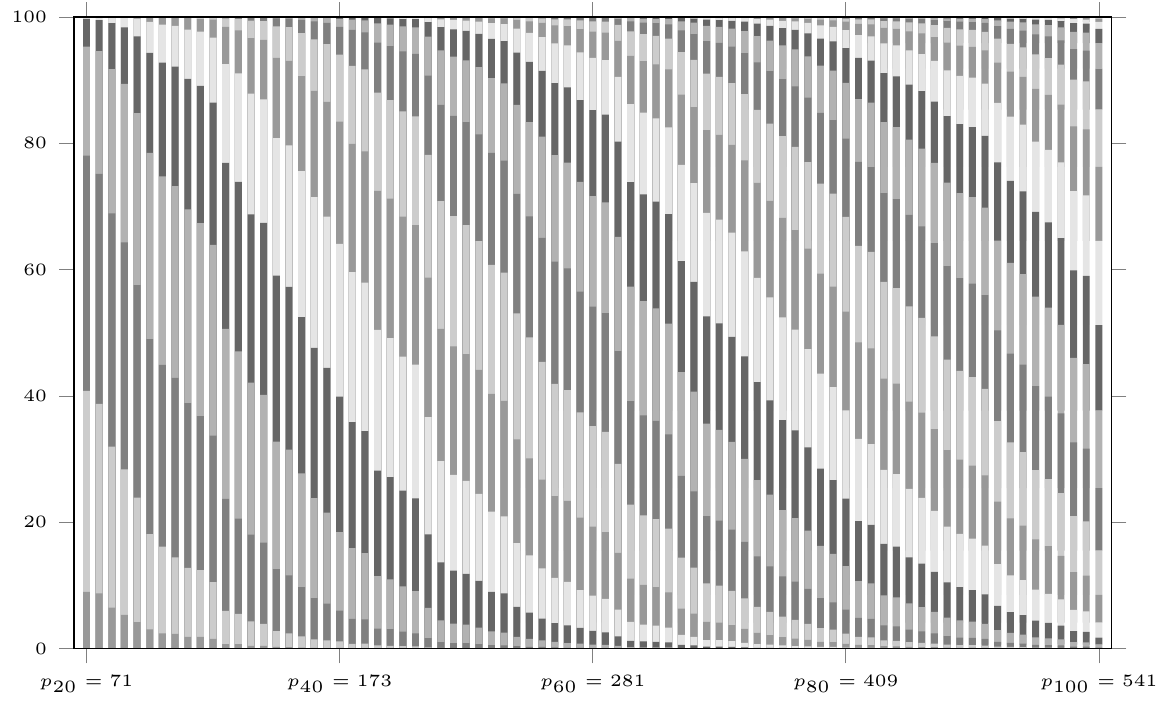}
\end{center}
\caption{Proportion of number of leaves when $n=5$}
\label{fig:prop5}
\end{figure}

\begin{figure}[H]
\begin{center}
\includegraphics{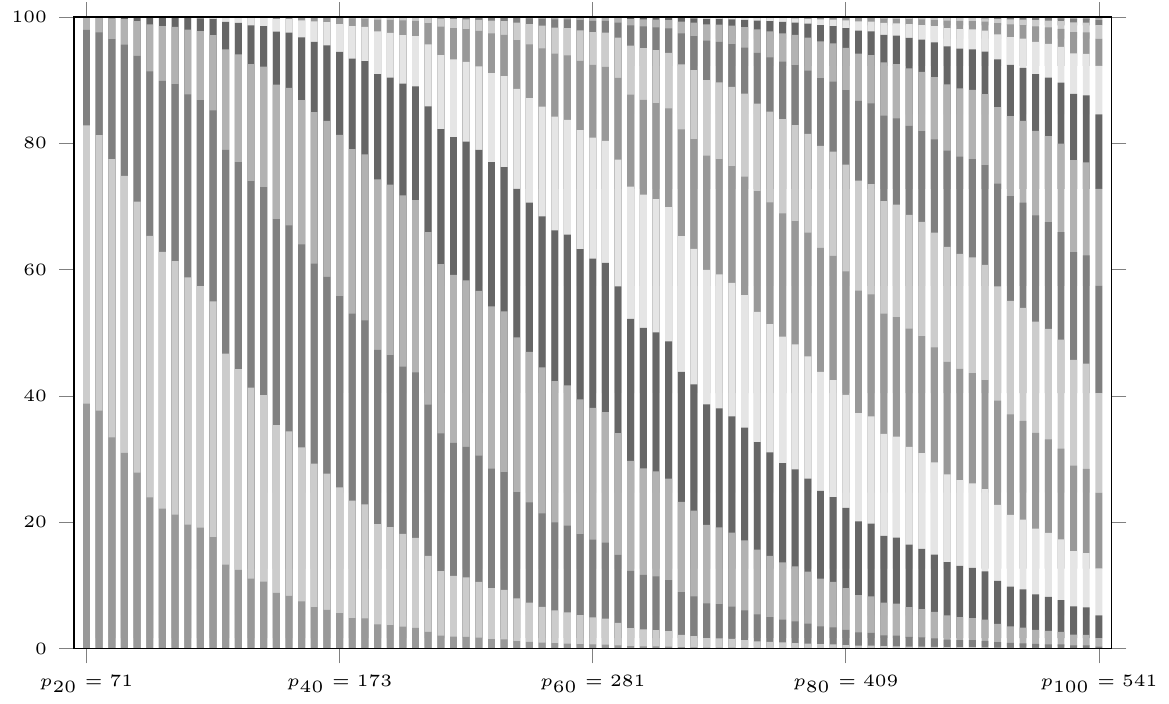}
\end{center}
\caption{Proportion of number of leaves when $n=6$}
\label{fig:prop6}
\end{figure}

\section{Comments}
The bounds of  Theorems~\ref{thm:Leaves} and~\ref{thm:NoLeaves} are of the same order 
of magnitude, and they essentially differ by a factor of $4$.
A rather naive heuristic model suggests that Theorems~\ref{thm:NoLeaves} is perhaps more precise. 
 It is certainly interesting to  investigate this in more details. 
 
 In Section~\ref{sec:numerical}, since the number of graphs we have to examine rapidly grows with the number of polynomials $n$ and the field size $q$, we also ask about further improvements of our  algorithms.

\section*{Acknowledgement}  
The authors would like to thank the referees for their careful reading and  valuable comments. 

For the research, B.~Mans and D.~Sutantyo were partly supported by the Australian Research Council (Discovery Project DP170102794), 
M.~Sha by the Guangdong Basic and Applied Basic Research Foundation (No. 2022A1515012032), 
and I.~E.~Shparlinski by the Australian Research Council (Discovery Projects DP180100201 and DP200100355).

\end{document}